\documentclass[twoside,12pt, leqno]{article}
\usepackage{amsmath,amscd,amsthm,amssymb,amsxtra,latexsym,epsfig,epic,graphics}
\usepackage[matrix,arrow,curve]{xy}
\usepackage{graphicx}
\usepackage{diagrams}
\usepackage[english]{babel}
\usepackage[alphabetic,lite]{amsrefs} 

\usepackage[OT2,OT1]{fontenc}
\newcommand\cyr{%
\renewcommand\rmdefault{wncyr}%
\renewcommand\sfdefault{wncyss}%
\renewcommand\encodingdefault{OT2}%
\normalfont
\selectfont}
\DeclareTextFontCommand{\textcyr}{\cyr}

\def\cprime{\char126}

\usepackage{hyperref}
\hypersetup{colorlinks=true,backref=true,citecolor=blue}

\usepackage[table]{xcolor}
\def\antiddot{\mathinner{\mkern1mu\raise1pt\vbox{\kern7pt\hbox{.}}\mkern2mu
        \raise4pt\hbox{.}\mkern2mu\raise7pt\hbox{.}\mkern1mu}}

\newcommand{\PP}{{\mathbb P}}

\newcommand{\ZZ}{{\mathbb Z}}

\newcommand{\s}{\mathcal}

\newcommand{\sB}{{\s B}}

\newcommand{\sF}{{\s F}}

\newcommand{\sO}{{\s O}}

\newcommand{\tensor}{\otimes}

\newcommand{\punkt}{\hspace{-.3ex}\raise.15ex\hbox to1ex{\Huge.}}

\DeclareMathOperator{\Hom}{Hom}

\DeclareMathOperator{\rank}{rank}

\newtheorem{theorem}{Theorem}[section]
\newtheorem{lemma}[theorem]{Lemma}
\newtheorem{proposition}[theorem]{Proposition}

\theoremstyle{definition}

\newtheorem{example}[theorem]{Example}

\def\bT{{\bf T}}
\def\bU{{\bf U}}

\def\b1{{\bf 1}}
\def\b1{{1^t}}
\def\o{{\emptyset}}

\def\PP{{\mathbb P}}

\def\cornerT#1{{T_{\Rsh \kern -1pt #1}}}
\def\Tate#1#2#3#4{{T_#1(#2,#3,#4)}}

\def\PPn{{\PP^{n_1} \times \cdots  \times \PP^{n_t}}}

\newcommand\rcell{\cellcolor{black!10}}

\makeatletter
\def\Ddots{\mathinner{\mkern1mu\raise\p@
\vbox{\kern7\p@\hbox{.}}\mkern2mu
\raise4\p@\hbox{.}\mkern2mu\raise7\p@\hbox{.}\mkern1mu}}
\makeatother

\usepackage{times}
\newdimen\x \x=12pt

\usepackage{color}

\date{\today}
\title{Horrocks splitting on Segre-Veronese varieties }

\author{Frank-Olaf Schreyer
}

\begin{document}

\maketitle

\begin{abstract} 
\end{abstract}
We prove an analogue of Horrocks' splitting theorem for Segre-Veronese varieties building upon the theory 
of Tate resolutions on products of projective spaces. 
\section*{Introduction}

Horrocks' famous splitting theorem \cite{horrocks} on $\PP^n$ says that a vector bundle $\sF$ on $\PP^n$ splits into a direct sum of line bundles
$$\sF \cong \oplus_j \sO(k_j)$$ if and only if $\sF$ has no intermediate cohomology, i.e. if
$$H^i(\PP^n,\sF(k))=0\;  \forall k \in \ZZ \hbox{ and } \forall i \hbox{ with } 0< i <n.$$
In this note we prove a similar criterion for Segre-Veronese varieties $$\PP^{n_1}\times \ldots \times \PP^{n_t} \hookrightarrow \PP^N$$ embedded
by the complete linear system of a very ample line bundle  $\sO(H)=\sO(d_1,\ldots,d_t)$, so $N= (\prod_{j=1}^t {n_j+d_j \choose n_j}) -1$. 

\begin{theorem}\label{main} Let $\sO(H) = \sO(d_1,\ldots,d_t)$ be a very ample line bundle on 
a product of projective spaces $\PP=\PP^{n_1}\times \ldots \times \PP^{n_t}$ of dimension $m=n_1 +\ldots+n_t$ with $t \ge 2$ factors.
A torsion free sheaf $\sF$ on $\PP$ splits into a direct sum $\sF\cong \oplus_j \sO(k_jH)$ if and only if
$$H^i(\PP,\sF(a_1,\ldots,a_t))=0\; \forall i \hbox{ with } 0<i<m$$
for all twists with $\sO(a_1,\ldots,a_t)$ such that  the  cohomology groups $H^i(\PP,\sO(kH)\tensor \sO(a_1,\ldots,a_t)) \hbox{ for all } i \hbox{ with } 0<i<m$ and all $k \in \ZZ $ vanish.
\end{theorem}

We can rephrase the theorem as follows: {\it If a torsion free sheaf on a product $\PPn$ has no intermediate cohomology in the range where the sheaves $\sO(kH)$ have no intermediate cohomology, then it is a direct sum of these sheaves.}

\begin{example} For $\PP=\PP^{n_1}\times \PP^{n_2}$ the line bundle $\sO(a_1,a_2)$ has nonzero cohomology for $a=(a_1,a_2)$ in the range 
$\{ a_1\ge 0, a_2\ge 0\}, \{a_1 < -n_1, a_2 \ge 0 \}, \{ a_1 \ge 0, a_2 < -n_2 \} \hbox { and }\{ a_1 < -n_1, a_2 < -n_2 \}$
where
$
H^0(\sO(a_1,a_2))\not=0,H^{n_1}(\sO(a_1,a_2))\not=0,H^{n_2}(\sO(a_1,a_2))\not=0,\hbox{ respectively } H^{m}(\sO(a_1,a_2))\not=0.
$

In particular for $\PP^{2}\times \PP^{3}$ and the range $\{-5 \le a_1 \le 1,-5 \le a_2 \le 2 \}$ nonzero cohomology and nonzero intermediate cohomology
occurs in the shaded areas 
$$
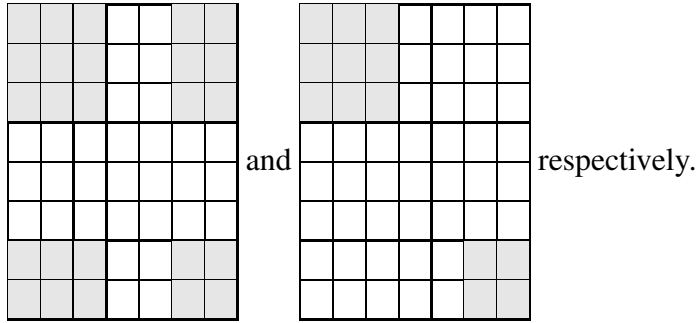
\begin{tabular}{|c|c|c|c|c|c|c|c|c}\hline
\rcell &\rcell & \rcell & & &\rcell & \rcell \cr\hline
\rcell &\rcell & \rcell & & &\rcell & \rcell \cr\hline
\rcell &\rcell & \rcell & & &\rcell & \rcell \cr\hline 
         &         &         & & &         &         \cr\hline
         &         &         & & &         &         \cr\hline
         &         &         & & &         &         \cr\hline 
\rcell &\rcell & \rcell & & &\rcell & \rcell \cr\hline
\rcell &\rcell & \rcell & & &\rcell & \rcell \cr\hline
\end{tabular} \hbox{ and }
\begin{tabular}{|c|c|c|c|c|c|c|c|c}\hline
\rcell &\rcell & \rcell & & & &  \cr\hline
\rcell &\rcell & \rcell & & & & \cr\hline
\rcell &\rcell & \rcell & & &   &  \cr\hline
         &         &         & & &         &         \cr\hline
         &         &         & & &         &         \cr\hline
         &         &         & & &         &         \cr\hline
         &         &         & & &\rcell & \rcell \cr\hline
          &        &         & & &\rcell & \rcell \cr\hline
\end{tabular}  \hbox{ respectively. }
$$
Thus for $\sO(H)=\sO(4,2)$ the assumption of the theorem  in this case is, that the intermediate cohomology occurs only in an area
 as indicated  below:
 $$
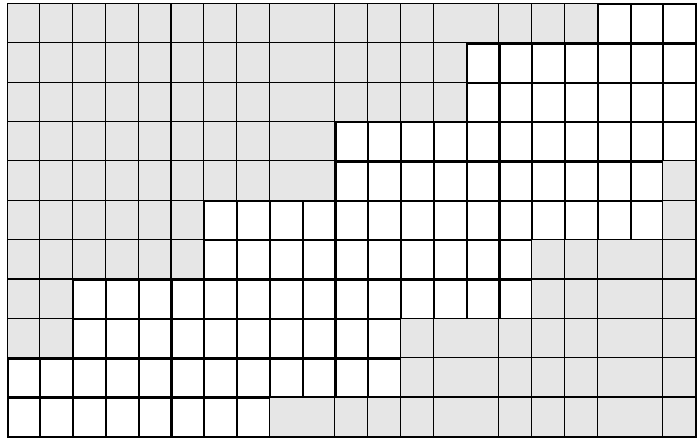
\begin{tabular}{|c|c|c|c|c|c|c|c|c|c|c|c|c|c|c|c|c|c|c|c|c|c|c|c|c|c|c|c|} \hline
\rcell &\rcell & \rcell & \rcell& \rcell&\rcell & \rcell &\rcell &\rcell & \rcell & \rcell& \rcell&\rcell & \rcell &\rcell &\rcell & \rcell &\rcell & & &  \cr\hline
\rcell& \rcell&\rcell & \rcell &\rcell &\rcell & \rcell & \rcell& \rcell&\rcell & \rcell &\rcell &\rcell & \rcell& & && & & &\cr \hline
\rcell& \rcell&\rcell & \rcell &\rcell &\rcell & \rcell  &\rcell& \rcell&\rcell & \rcell &\rcell &\rcell & \rcell& & && & & &\cr \hline
\rcell& \rcell&\rcell & \rcell &\rcell &\rcell & \rcell  &\rcell& \rcell&\rcell &  & & & & & & && & & \cr \hline
\rcell& \rcell&\rcell & \rcell &\rcell &\rcell & \rcell  &\rcell& \rcell&\rcell &  & & & & & & && & &\rcell \cr \hline
\rcell& \rcell&\rcell & \rcell &\rcell &\rcell & & & & &  & & & & & & && & &\rcell \cr \hline
\rcell& \rcell&\rcell & \rcell &\rcell &\rcell & & && &  & & & & &  & \rcell & \rcell &\rcell &\rcell&\rcell \cr \hline
\rcell& \rcell& & & & & & & & &  & & & & &  & \rcell & \rcell &\rcell &\rcell&\rcell \cr \hline
\rcell& \rcell& & & & & & & & &  & &  \rcell & \rcell &\rcell &\rcell & \rcell & \rcell &\rcell &\rcell&\rcell \cr \hline
& & & & & & & & & &  & &  \rcell & \rcell &\rcell &\rcell & \rcell & \rcell &\rcell &\rcell&\rcell \cr \hline
& & & & & & & &  \rcell & \rcell &\rcell &\rcell &  \rcell & \rcell &\rcell &\rcell & \rcell & \rcell &\rcell &\rcell&\rcell \cr \hline
\end{tabular}
$$
\end{example}
\section{Preliminaries and Notation}

 The Tate resolutions of a sheaf on products of projective spaces in a generalization of the Tate resolution on $\PP^n$ \cite{EFS}.
 We recall from \cite{EES} the basic notation.

Let $\PP=\PPn =  \PP(W_{1})\times\cdots\times \PP(W_{t})$ be a product of $t$ projective spaces over an arbitrary field $K$. Set $V_{i} = W_{i}^{*}$ and  $V = \oplus_{i}V_{i}$. Let $E$ be the $\ZZ^{t}$-graded exterior algebra on $V$, where elements of $V_{i}\subset E$ have degree $(0,\dots,0, -1, 0,\dots,0)$ with $-1$ in the $i$-th place.

For a sheaf $\sF$ on $\PP$ the Tate resolution $\bT(\sF)$ is a minimal exact complex of graded $E$-modules with terms 
$$
\bT(\sF)^d = \oplus_{a\in \ZZ^t} \Hom_K(E,H^{d-|a|}(\PP, \sF(a))),
$$
where the cohomology group $H^{d-|a|}(\PP,\sF(a))$ is regarded as a vector
space concentrated in degree $a$ and $|a|= \sum_{j=1}^t a_j$ denotes the total degree.

Since $\omega_E = \Hom_K(E,K)$ is the free $E$-module of rank 1 with socle in degree $0$ and hence generator in degree 
$(n_1+1,\ldots,n_t+1)$, the differential of the complex $\bT(\sF)$ is given by a matrix with entries in $E$. More precisely,
the component $\Hom_K(E,H^{d-|a|}(\PP, \sF(a))) \to \Hom_K(E,H^{d+1-|b|}(\PP, \sF(b)))$ is given by a
$h^{d+1-|b|}(\PP, \sF(b)) \times h^{d-|a|}(\PP, \sF(a))$-matrix with
entries in 
$$ \Lambda^{b-a}V := \Lambda^{b_1-a_1} V_1 \tensor \ldots \tensor \Lambda^{b_t-a_t} V_t.$$
In particular, all blocks corresponding to cases where  $b_j< a_j$ for some $j$ are zero. Moreover, all blocks to cases with $a=b$ are also zero, since $\bT(\sF)$ is a minimal complex.

The complex $\bT(\sF)$ has various exact free subquotient complexes: For $c\in \ZZ^t$ a degree and $I, J, K \subset \{1,\ldots, t \}$ disjoint subsets we have the subquotient complex $T_c(I,J,K)$ with terms
$$
\Tate c I J K ^d= \sum_{a \in \ZZ 
          \atop {a_i < c_i \hbox{ \scriptsize for } i \in I
         \atop {a_i  = c_i \hbox{ \scriptsize for } i \in J
         \atop a_i  \ge c_i \hbox{ \scriptsize for } i \in K}}} 
          \Hom_K(E,H^{d-|a|}(\PP,\sF(a)))
$$
By \cite[Theorem 3.3 and Corollary 3.5]{EES} these complexes are exact as long as $I \cup J \cup K \subsetneq \{1,\ldots,t \}$. 
The complexes $\Tate c \o J \o$ can be used to compute the direct image  complex of $\sF(c)$ along a partial projection
$\pi_J \colon \PP \to \prod_{j \notin J } \PP^{n_j}$ \cite[Corollary 0.3 and Proposition 3.6]{EES}.

\begin{lemma}\label{on strands} Let $\sF$ be a coherent sheaf on a product of projective spaces $\PP= \PP^{n_1} \times \ldots \times \PP^{n_t}$ and let
$a=(a_1,a_2,\ldots, a_n)=(a',a_t) \in \ZZ^t= \ZZ^{t-1} \times \ZZ^{}$ and $n\in \ZZ$.
If $$H^n(\sF(a',a_t))=H^{n-1}(\sF(a',a_t+1))=\ldots=H^{n-n_t}(\sF(a',a_t+n_t))=0$$   then
$H^n(\sF(a',a_t-1))=0$ as well. 
A similar statement holds for the cohomology along the $j$-th strand $T_a(\o,{\{1,\ldots,t\}\setminus \{j\}},\o)$.
\end{lemma}

\begin{proof} We consider the strand $T_a(\o,\{1,\ldots,t-1\},\o)$ of $\bT(\sF)$. The differential
starting at the summand $\Hom_K(E,H^n(\sF(a',a_t-1)) \subset T_a(\o,\{1,\ldots,t-1\},\o)$ maps in the strand  to the summands 
$$
\Hom_K(E, H^n(\sF(a',a_t)))\oplus  \ldots \oplus \Hom_K(E,H^{n-n_t}(\sF(a',a_t+n_t))).
$$ 
 By assumption the target is zero. Since of $T_a(\o,\{1,\ldots,t-1\},\o)$ is minimal and exact, the source
is zero as well. 
\end{proof}

The proof of the our main theorem uses the corner complexes $T_{\Rsh \kern -1pt c}(\sF)$
which are defined as the cone of a map of complexes
$$
\varphi_{c}: T_c(\{1,\ldots,t\},\emptyset,\emptyset)[-t] \to T_c(\emptyset,\emptyset,\{1,\ldots,t\}) 
$$
obtained as the composition of $t$ maps
$$
T_c(\{1,\ldots,k\},\emptyset,\{k+1,\ldots,t\})[-k] \to  T_c(\{1,\ldots,k-1\},\emptyset,\{k,\ldots,t\})[-k+1] 
$$
each of which is obtained from the differential of $\bT(\sF)$ by taking the terms with source in one quadrant and target in the next quadrant. The corner complexes are exact as well by \cite[Theorem 4.3 and Corollary 4.5]{EES}. 

If we follow a path from the last quadrant to the first quadrant using a different order of the elements in the set 
$\{1,\ldots,t \}$, we obtain an isomorphic complex, since both walk around the corner complexes coincide
in sufficiently large cohomological degree $d$.

\section{Proof of the main result}

We use the partial order $a \ge b$ on $\ZZ^t$ defined by $a_j \ge b_j \hbox{ for } j=1,\ldots, t$ and write $a > b$ if $a \ge b$ and $a \not=b$.

Let $\sF$ be a coherent sheaf on $\PP=\PPn$.
If $H^m(\PP,\sF(a)) \not=0$ then $H^m(\PP,\sF(b)) \not=0$ for all $b \le a$ as we see from applying $H^m$ to the surjection 
$$ H^0(\PP,\sO(a-b)) \otimes \sF(b) \to \sF(a).$$

 An {\it extremal $H^m$-position} of $\sF$ is an degree $a\in \ZZ^t$ such that
$H^m(\PP,\sF(a)) \not= 0$ but $H^m(\PP,\sF(c))=0$ for all $c > a$.

\begin{proposition}\label{extremal Hm} Let $\sF$ be a torsion free sheaf on $\PPn$ satisfying the assumption of 
Theorem \ref{main} with respect to
$\sO(H) = \sO(d_1,\ldots,d_t)$. There exists an extremal $H^m$-position for $\sF$ of the form
$$ (a_1,\ldots,a_t)=(kd_1-n_1-1,\ldots,kd_t-n_t-1) $$
for some $k \in \ZZ$.
\end{proposition}

Note that $\sO(-n_1-1,\ldots,-n_t-1) \cong \omega_\PP$ is the canonical sheaf on $\PP$.

\begin{proof} Since $\sF$ is nonzero and torsion free, we have
$H^m(\sF(kH)\tensor \omega_\PP) \not=0$ for $k\ll 0$ and $H^m(\sF(kH) \tensor \omega_\PP) =0$ for $k \gg 0$.
Let $k$ be the maximum such that $H^m(\sF(kH)\tensor \omega_\PP) \not=0$. We claim that this is an extremal $H^m$-position.
Suppose it is not. Then there exists a maximal $a$ in the range
$$(kd_1-n_1-1, \ldots,kd_t-n_t-1) < a \le (k+1)d_1-n_1-1, \ldots,(k+1)d_t-n_t-1)$$
such that $H^m(\sF(a)\tensor \omega_\PP) \not=0$. At least for one $i$ we have $kd_i-n_i-1 < a_i$. Then for $j \not=i$ we consider $J=\{1,\ldots,t\}\setminus \{j\}$ and look at the $j$-th
strand $\Tate a \o J \o $ through $a$.  Lemma \ref{on strands} implies $a_j=(k+1)d_j-n_j-1$: If  $a_j<(k+1)d_j-n_j-1$,
then we cannot reach the intermediate cohomology range of $\sF$ after at most $n_j+1$ steps along this strand, contradicting  the extremality of $a$. Interchanging the role of $i$ and $j$, we get
 $a_i=(k+1)d_i-n_i-1$ for all $i$, which contradicts the maximality of $k$.
\end{proof}

\begin{proposition}\label{summand}  Let $\sF$ be a torsion free sheaf on $\PPn$ satisfying the assumption of Theorem \ref{main} with respect to
$\sO(H) = \sO(d_1,\ldots,d_t)$. If $$(kd_1-n_1-1,\ldots,kd_t-n_t-1)$$ is an extremal $H^m$-position for $\sF$, then
$$ \sF\cong \sO(kH) \oplus \sF'.$$
\end{proposition}

\begin{proof}  We consider the corner complex $T_{\Rsh \kern -1pt c}(\sF)$ for $c=(kd_1-n_1,\ldots,kd_t-n_t)$. The first part of the corner map
with source $\Hom_K(E,H^m(\sF(kH) \tensor \omega_\PP)$ is a map
$$
\Hom_K(E,H^m(\sF(kH) \tensor \omega_\PP)) \to \Hom_K(E,H^{m-n_t}(\sF\tensor \omega_\PP \tensor \sO(0,\ldots,0,n_t+1)))
$$
given by a matrix with entries in $\Lambda^{n_t+1} V_t$. Since $\Lambda^{n_t+2} V_t=0$, only the map to
$$
 \Hom_K(E,H^{m-n_t-n_{t-1}}(\sF\tensor \omega_\PP \tensor \sO(0,\ldots,0,n_{t-1}+1,n_t+1)))
$$
can be nonzero for the composition of the first two parts. Repeating this arguments, we conclude that the corner map with source 
$\Hom_K(E,H^m(\sF(kH) \tensor \omega_\PP)$ has an image only in $\Hom_K(E,H^0(\sF(kH))$. 
It is given by an
$$h^0(\sF(kH)) \times h^m(\sF(kH) \tensor \omega_\PP)\hbox{-matrix}$$ with entries in the one-dimensional space
$$\Lambda^{m+t} V =\Lambda^{n_1+1} V_1 \tensor \ldots \tensor \Lambda^{n_t+1} V_t.$$
Consider the submatrix of the differential in the corner complex with target equal to the summand $\Hom_K(E,H^0(\sF(kH))$. 
The only other subspaces in the source which have this target come from $H^0$-groups:
$$\Hom_K(E,H^0(\sF(kH)(-1,0,\ldots,0)), \ldots, \Hom_K(E,H^0(\sF(kH)(0,\ldots,0,-1)).$$
Thus this differential is given by an
$$h^0(\sF(kH)) \times [(h^m(\sF(kH)\tensor \omega_\PP) + h^0(\sF(kH)\otimes \sB) ) ]\hbox{-matrix} $$
with $\sB= \sO(-1,0,\ldots,0) \oplus \ldots \oplus \sO(0,\ldots,0,-1)$.
Note that $h^0(\sF(kH)) \ge h^m(\sF(kH)\tensor \omega_\PP)$, because otherwise a generator of 
$\Hom_K(E,H^m(\sF(kH) \tensor \omega_\PP)$ would map to zero which is impossible because $T_{\Rsh \kern -1pt c}(\sF)$ is exact and minimal. Thus in a suitable basis the matrix has shape
$$
\varphi=\begin{pmatrix}
v &         & 0 &\vrule &\ell_{1j} &\ldots & \ell_{1n}   \cr
   &\ddots &  &\vrule& \vdots   & & \vdots \cr
0 & & v & \vrule& \ell_{rj} & \ldots & \ell_{rj}   \cr \hline
  & &  &  \vrule&\ell_{r+1j} & \ldots & \ell_{r+1n}   \cr
 &0 & &   \vrule&\vdots   & & \vdots \cr
  & &  &  \vrule&\ell_{sj} & \ldots & \ell_{sn}   \cr
\end{pmatrix}
$$
with $v \in \Lambda^{m+t} V$ a fixed basis element and $ \ell_{ij} \in V_1 \cup \ldots \cup V_t$.

We claim now that $\ell_{1j}$ is a $K$-linear combination of $\ell_{r+1j}, \ldots, \ell_{sj}$. Indeed if not, we could multiply
the $j$-th column by an element $w \in \Lambda^{m+t-1} V$ which annihilates $\ell_{r+1j}, \ldots, \ell_{sj}$ such that $\ell_{1j}w =v$.
This would give us a column
$$
\begin{pmatrix}
v &         & 0 &\vrule &v  \cr
   &\ddots &  &\vrule& \vdots    \cr
0 & & v & \vrule& \lambda_r v &   \cr \hline
  & &  &  \vrule&0   \cr
 &0 & &   \vrule&\vdots    \cr
  & &  &  \vrule&0   \cr
\end{pmatrix}\
$$
for possibly zero scalars $\lambda_2, \ldots, \lambda_r$, and the first column would be an $E$-linear combination of columns $2$ to $j$.
This is impossible since no generator can map to zero in $T_{\Rsh \kern -1pt c}(\sF)$.
So after row operations we may assume that $\varphi$ has the shape
$$
\varphi=\begin{pmatrix}
v &         & 0 &\vrule &0 &\vrule&\ldots & \ell_{1n}   \cr \hline
   &\ddots &  &\vrule& \vdots&\vrule   & & \vdots \cr
0 & & v & \vrule& \ell_{rj} &\vrule& \ldots & \ell_{rn}   \cr \hline
  & &  &  \vrule&\ell_{r+1j} &\vrule& \ldots & \ell_{r+1n}   \cr
 &0 & &   \vrule&\vdots &\vrule  & & \vdots \cr
   & &  &  \vrule&\ell_{r_1j} &\vrule& \ldots & \ell_{r_1n}   \cr \hline
      & & &   \vrule&\vdots &\vrule & & \vdots \cr
  & &  &  \vrule&0 &\vrule& \ldots & \ell_{sn}   \cr
\end{pmatrix}
$$
with $\ell_{r+1j},\ldots \ell_{r_1j}$ $K$-linearly independent. 

Next we note that the columns of the matrix
$$
\begin{pmatrix}
v &         & 0 &\vrule &  &     &    &  \vrule& \ell_{1j+1}   \cr 
   &\ddots &  &\vrule&    & 0& & \vrule&\vdots    \cr
0 & & v & \vrule&    & &  &\vrule& \vdots   \cr \hline
  & &  &  \vrule&  v  & & 0& \vrule&\vdots   \cr
 &0 & &   \vrule&    & \ddots& & \vrule& \vdots \cr
  & &  &  \vrule&  0  & & v & \vrule&\ell_{r_1j+1}    \cr \hline
      & &  &  \vrule&    & &  & \vrule&\ell_{r_1+1j+1}     \cr 
    & 0&  &  \vrule&    & 0&  & \vrule&\vdots    \cr 
      & &  &  \vrule&    & & & \vrule&\ell_{sj+1}    \cr 
\end{pmatrix}
$$
are in the $E$-column span of $\varphi$. Arguing as before, we see that $\ell_{1j+1}$ is a linear combination of $\ell_{r_1+1j+1},\ldots,\ell_{sj+1}$, and repeating the arguments, we find that $\varphi$ can be transformed by row operations into a matrix of type
$$\begin{pmatrix}
v&\vrule &         & 0 & \vrule &0 &\ldots & 0   \cr \hline
 &\vrule&   \ddots      &   & \vrule&\ell_{2j} &\ldots & \ell_{2n}   \cr
0 &\vrule& & v & \vrule &\vdots&  & \vdots  \cr \hline
 &\vrule & &  & \vrule&\vdots & & \vdots   \cr
  &\vrule& & 0 & \vrule&\ell_{sj} & \ldots & \ell_{sn}   \cr
\end{pmatrix} = 
\begin{pmatrix}
v &\vrule& 0 \cr \hline
0  &\vrule& \varphi'   \cr
\end{pmatrix}  
$$
We conclude that $T_{\Rsh \kern -1pt c}(\sO(kH))$ is a direct summand of the complex $T_{\Rsh \kern -1pt c}(\sF)$, and
$$ \sF \cong \sO(kH)\oplus \sF',$$
since we can recover $\sF$ from its corner complex with the Beilinson functor $\bU$ applied to  $T_{\Rsh \kern -1pt c}(\sF)(a)[|a|]$
for a suitable $a \in \ZZ^t$ by \cite[Theorem  0.1]{EES}. Indeed $\bU(T_{\Rsh \kern -1pt c}(\sF)(a)[|a|])$ and $\bU(\bT(\sF)(a)[|a|])$ coincide for
$ a\gg 0$.
\end{proof}

{\it Proof} of Theorem \ref{main}. Let $\sF$ be a torsion free sheaf on $\PPn$ with no intermediate cohomology where the sheaves $\sO(kH)$ for 
$\sO(H)=\sO(d_1,\ldots,d_t)$ have no intermediate cohomology. By Proposition \ref{extremal Hm} there is an extreaml $H^m$-position of $\sF$ of the form
$$(k_1d_1-n_1-1,\ldots,k_1d_t-n_t-1)$$
and by Proposition \ref{summand} we get a summand
$$\sF \cong \sO(k_1H) \oplus \sF'.$$
If $\rank \sF=1$, we are done: $\sF'=0$ since $\sF$ is torsion free. Otherwise we can argue by induction on the rank since $\sF'$ satisfies
the assumption of the Theorem again. 
\qed

\noindent {\bf Acknowledgement} I thank  Prabhakar Rao for very valuable discussions on the material of this paper.

\noindent{Frank-Olaf Schreyer}\par
\noindent{Mathematik und Informatik, Universit\"at des Saarlandes, Campus E2 4, \\ 
D-66123 Saarbr\"ucken, Germany}\par
\noindent{schreyer@math.uni-sb.de}\par

\end{document}